\documentclass[11pt]{amsart}
\usepackage{amsfonts}
\usepackage{amsfonts,latexsym,rawfonts,amsmath,amssymb,amsthm}
\usepackage[plainpages=false]{hyperref}

\usepackage{graphicx}

\RequirePackage{color}

\theoremstyle{plain}
  \newtheorem{theorem}{Theorem}[section]
  \newtheorem{lemma}{Lemma}[section]

  \newtheorem{corollary}{Corollary}[section]

\theoremstyle{remark} 
  \newtheorem{remark}{Remark}[section]

\renewcommand{\det}{\mbox{det}}
  
  \newcommand{\dist}{\mbox{dist}}

  \numberwithin{equation}{section}
  \numberwithin{figure}{section}

\begin{document}

\title[An obstacle problem for Monge-Amp\`ere type functionals]
{An obstacle problem for a class of Monge-Amp\`ere type functionals}

\date\today

\author[J. Liu\and B. Zhou]
{Jiakun Liu\and Bin Zhou}

\address{Department of Mathematics\\
	Princeton University\\
	Fine Hall, Washington Road\\
	Princeton, NJ 08544-1000, USA.}
\email{jiakunl@math.princeton.edu}

\address{School of Mathematical Sciences
 and Beijing International Center for Mathematical Research\\
	Peking University\\
	Beijing 100871, China.}
\email{bzhou@pku.edu.cn}

\thanks{The first author is supported by the Simons Foundation. The second author is supported by National Science Foundation of China No. 11101004 and China Postdoctoral Science Foundation.}
\thanks{\copyright 2012 by the authors. All rights reserved}



\begin{abstract}
In this paper we study an obstacle problem for Monge-Amp\`ere type functionals, 
whose Euler-Lagrange equations are a class of fourth order equations, including the affine maximal surface equations and Abreu's equation.
\end{abstract}

\maketitle

\baselineskip=16.4pt
\parskip=3pt

\section{Introduction}\label{s1}

Free boundary and obstacle problems for partial differential equations 
have been studied extensively in the past decades.
For Monge-Amp\`ere equations, obstacle problems 
were studied in \cite{CW,Lee,Sa} among others, and a related free boundary problem 
was studied in \cite{CM}. In this paper we consider an obstacle problem for 
the functional
   \begin{equation}\label{e101}
   J_\alpha(u)=\begin{cases}
                       \int_\Omega \left[\det\,D^2u\right]^{\alpha}
                       -\alpha\int_\Omega fu, &\ \alpha>0 \ 
\text{and}\ \alpha\neq 1, \\[5pt]
		       \int_\Omega \log\det\,D^2u -\int_\Omega fu, &\ \alpha=0,
                      \end{cases}
  \end{equation}
where $\Omega$ is a bounded domain in $\mathbb{R}^n$ and $f\in L^\infty(\Omega)$. For simplicity, we denote the nonlinear part 
of the functional \eqref{e101} by $A_\alpha(u)$, see \eqref{e204}.
We would like to study the maximization problem
  \begin{equation}\label{e102}
   J_\alpha(u)=\sup\left\{J_\alpha(v)\,:\,v\in\mathcal{S}[\varphi,\psi]\right\},
  \end{equation}
where $\mathcal{S}[\varphi,\psi]$ is the class of functions
  \begin{equation}\label{e103}
   \mathcal{S}[\varphi,\psi]=\left\{u\in C(\overline\Omega)\,:\,u \mbox{ convex }, u|_{\partial\Omega}=\varphi, Du(\Omega)\subset D\varphi(\overline\Omega), u\geq\psi \mbox{ in }\Omega\right\},
  \end{equation}
$\varphi$ is a smooth, uniformly convex function defined on a neighborhood of $\overline\Omega$, $\psi$ is an obstacle function, and $Du(\Omega)$ represents the image of the subgradients of $u$ at all points $x\in\Omega$. 

The Euler-Lagrange equations of \eqref{e101} are a class of fourth order equations, that is,
  \begin{equation}\label{e104}
   U^{ij}w_{ij}=f,
  \end{equation}
where $(U^{ij})$ is the cofactor matrix of the Hessian $D^2u$, and
  \begin{equation}\label{e105}
   w=\left[\det\,D^2u\right]^{-(1-\alpha)},\ \ \alpha\geq 0.
  \end{equation}
When $\alpha=\frac{1}{n+2}$, equation \eqref{e104} is the affine mean curvature equation and the functional \eqref{e204} is the affine area functional. When $\alpha=0$, equation \eqref{e104} is Abreu's equation arising from the study 
of Calabi's extremal metrics on toric K\"ahler manifolds \cite{D1, D2, D3, D4}. 

Due to their importance in geometry, variational problems of \eqref{e101} have attracted much
interest in recent years. In the case of $\alpha=\frac{1}{n+2}$, the variational problem without obstacle is the graph case of affine Plateau problem \cite{TW08, TW10}, raised by
Calabi and Chern. The case of $\alpha=0$ has been treated in \cite{Z1}.
The obstacle problem of affine maximal surfaces was first introduced in \cite{STW}.
In this paper, we obtain:

\begin{theorem}\label{t101}
Suppose $n=2$, $0\leq\alpha\leq \frac{1}{n+2}$, and $f\in L^\infty(\Omega)$. 
 Let $\varphi$ be a smooth, uniformly convex function in $\Omega$. 
  If $\psi$ is a convex function in $\Omega$ satisfying $\psi<\varphi$ 
  on $\partial\Omega$, then there exists a unique maximizer of \eqref{e102} 
  which is strictly convex and $C^{1,\alpha}$ in $\Omega$.
Furthermore, if $\psi$ is uniformly convex $\Omega$, 
then the maximizer of \eqref{e102} is $C^{1,1}$ in $\Omega$. 
\end{theorem}

We remark that in higher dimensions, the problem is more complicated since Lemma \ref{l401} does not hold. Furthermore, in the case of $\alpha=0$, the interior estimate in Lemma \ref{l203} remains open when $n>2$. We will consider the higher dimensional cases and more general forms of the Monge-Amp\`ere type functionals with $f=f(x,u,Du)$ is our forthcoming work.

This paper is organized as follows: In Section 2 we recall some preliminary results that will be used in subsequent sections. In addition, we show that how the functionals and equations change under a rotation in $\mathbb{R}^{n+1}$ and obtain the a priori determinant estimates under the rotation transform, where the functionals have more general forms \eqref{e215}. 
In Section 3 we show that the maximizer of $J_\alpha$ can be approximated by a sequence of smooth maximizers of appropriate penalized functionals. 
In Section 4 we prove that the maximizer is strictly convex by an observation in \cite{TW05,TW10}. 
The proof of Theorem \ref{t101} is contained in Section 5, where the $C^{1,\alpha}$ and $C^{1,1}$ regularities are obtained, respectively.

\section{Preliminaries}

\subsection{Monge-Amp\`ere measure}

Let $\Omega$ be a bounded domain in $\mathbb{R}^n$ and 
$u$ be a convex function in $\Omega$.
The {\it normal mapping} of $u$, $N_u$, is a set-valued mapping defined as follows.
For any point $x\in\Omega$, $N_u(x)$ is the set of slopes of supporting hyperplanes of $u$ at $x$, that is,
  \begin{equation}\label{e201}
   N_u(x)=\{p\in\mathbb{R}^n\,:\,u(y)\geq u(x)+p\cdot(y-x),\quad\forall y\in\Omega\}.
  \end{equation}
For any Borel set $E\subset\Omega$, $N_u(E)=\bigcup_{x\in E}N_u(x)$. If $u$ is $C^1$, the normal mapping $N_u$ is exactly the gradient mapping $Du$. 

From the normal mapping we define the {\it Monge-Amp\`ere measure} $\mu[u]$ by
  \begin{equation}\label{e202}
   \mu[u](E)=|N_u(E)|
  \end{equation}
for any Borel set $E\subset\Omega$, where the right hand side is the Lebesgue measure of $N_u(E)$. 
If $u$ is $C^2$ smooth, we have $\mu[u]=(\det\,D^2u)dx$. In the non-smooth case, the Monge-Amp\`ere measure $\mu[u]$ is a Radon measure, and is weakly continuous with respect to the convergence of convex functions, namely if a sequence of convex functions $\{u_i\}$ converges to a convex function $u$ in $L^\infty_{loc}$, then for any closed $E\subset\Omega$,
  \begin{equation}\label{e203}
   \lim_{i\to\infty}\sup\mu[u_i](E)\leq\mu[u](E).
  \end{equation}

\subsection{Existence and uniqueness of maximizer}

Note that the functional $J_\alpha$ in \eqref{e101} is well defined on the set of $C^2$-smooth, convex functions.  To study the maximization problem, 
we extend the functional $J_\alpha$ to the set
$\mathcal{S}[\varphi,\psi]$ in \eqref{e103}, which is closed under the locally uniform convergence of convex functions.
It is clear that the linear part
in $J_\alpha$ is naturally defined. It suffices to
extend the nonlinear part $A_\alpha$ to $\mathcal{S}[\varphi,\psi]$.
If $u$ is a convex function, $u$ is almost everywhere twice-differentiable,
i.e., the Hessian matrix $(D^2u)$ exists almost everywhere.
Denote the extended Hessian matrix by $\partial^2u(x)=D^2u(x)$ when $u$ is
twice differentiable at $x\in\Omega$ and $\partial^2u(x)=0$ otherwise.
As a Radon measure, $\mu[u]$ can be decomposed into
a regular part and a singular part as follows,
$$\mu[u]=\mu_r[u]+\mu_s[u].$$
It was proved in \cite{TW05} that the regular part $\mu_r[u]$
can be given explicitely by $\mu_r[u]=\det\, \partial^2u\, dx$
and hence $\det\, \partial^2u$ is a locally integrable function.
Therefore for any $u\in{ S[\varphi,\psi]}$, we can define
  \begin{equation}\label{e204}
   A_\alpha(u)=\left\{\begin{array}{ll}
                       \int_\Omega \left[\det\,\partial^2u\right]^{\alpha}, & \alpha>0, \\[5pt]
		       \int_\Omega \log\det\,\partial^2u, & \alpha=0.
                      \end{array}\right.
  \end{equation}

\begin{lemma}\label{l201}
 Suppose $0\leq\alpha\leq \frac{1}{n+2}$.
 $J_\alpha$ is upper semi-continuous, bounded and 
 concave in $\mathcal{S}[\varphi,\psi]$. 
 It follows that there exists a unique maximizer $u_0$ of \eqref{e102}. 
\end{lemma}

\begin{proof}
 The proof for the cases $\alpha=\frac{1}{n+2}$ and $\alpha=0$ can be found in \cite{TW05, Z2}, respectively. One can check that the proof also holds for 
 $0<\alpha<\frac{1}{n+2}$.
\end{proof}

\subsection{Estimates for classical solutions}

We include the following a priori estimates in \cite{TW00,TW05}, which will be needed in subsequent sections, 
see also \cite{D2, Z1} for the case of $\alpha=0$.
Consider the equation
	\begin{eqnarray}
		U^{ij}w_{ij}\!\!&=&\!\!f\quad\mbox{in }\Omega,\label{e205} \\
		w\!\!&=&\!\![\det\,D^2u]^{\alpha-1},\nonumber
	\end{eqnarray}
where $(U^{ij})$ is the cofactor matrix of the Hessian matrix $D^2u$, and $\alpha\in[0,1)$ is a constant.

\begin{lemma}\label{l202}
Let $u\in C^4(\Omega)\cap C^{0,1}(\overline\Omega)$ be a convex solution of \eqref{e205} with $u=0$ on $\partial\Omega$. 
Then for any $y\in\Omega$, we have the a priori estimate
	\begin{equation}\label{e206}
		\det\,D^2u(y)\leq C,
	\end{equation}
where $C$ depends only on $n$, $\alpha$, $\dist(y,\partial\Omega)$, 
$\sup_\Omega(-u)$, $\sup_\Omega|Du|$, and $\sup_\Omega f$.
\end{lemma}

\begin{remark}
In Lemma \ref{l202}, the constant $C$ is independent of $\inf_\Omega f$. Hence it is independent of $f$ if $f\leq0$.
By Lemma \ref{l302}, the maximizer $u_0$ of $J_\alpha$ can be locally approximated by smooth solutions of \eqref{e205}, and thus Lemma \ref{l202} still holds for non-smooth maximizers. When $\alpha=\frac{1}{n+2}$, the estimate \eqref{e206} was previously proved in \cite{TW05}.
\end{remark}

\begin{remark}
If $n=2$, the assumption $u=0$ on $\partial\Omega$ in Lemma \ref{l202} can be removed \cite{TW05}.
\end{remark}

To prove that $\det\,D^2u$ has a positive lower bound, we consider the Legendre transform $u^*$ of $u$, which is a convex function defined in the domain $\Omega^*=N_u(\Omega)$, given by
	\begin{equation}\label{e207}
		u^*(y)=\sup\{x\cdot y-u(x)\,:\,x\in\Omega\}.
	\end{equation}
If $u$ is strictly convex near $\partial\Omega$, $u$ can be recovered from $u^*$ by the same transform. If $u$ is $C^2$ smooth at $x$, $y=Du(x)$ and $\det\,D^2u(x)\neq0$, then the Hessian matrix $D^2u(x)$ is the inverse of the Hessian matrix $D^2u^*(y)$, and
	\begin{equation}\label{e208}
		\det\,D^2u(x)=[\det\,D^2u^*(y)]^{-1}.
	\end{equation}
In particular, if $u$ is a maximizer of the functional $J_\alpha$, $u^*$ is a maximizer of the dual functional
	\begin{equation}\label{e209}
		J^*_\alpha(u)=
		\begin{cases}
\ \int_{\Omega^*}[\det D^2u^*]^{1-\alpha}\, dy 
-\alpha\int_{\Omega^*}f(Du^*)(yDu^*-u^*)\det\,D^2u^*\,dy, \ \ \alpha>0 \ 
\text{and}\ \alpha\neq 1,&\\[5pt]
\ -\int_{\Omega^*}\det D^2u^*\log\det D^2u^*\,dy
-\int_{\Omega^*}f(Du^*)(yDu^*-u^*)\det\,D^2u^*\,dy,\ \ \alpha =0.&
\end{cases}
	\end{equation}
Therefore, if $u^*$ is smooth, it satisfies the equation
	\begin{equation}\label{e210}
		U^{*ij}w^*_{ij}=
		\begin{cases}
		-\frac{\alpha}{1-\alpha}f(Du^*)\det D^2u^*, &\ \alpha>0 \
\text{and}\ \alpha\neq 1,\\[5pt]
   -f(Du^*)\det D^2u^*, &\ \alpha=0,
		\end{cases}
	\end{equation}
where $U^{*ij}$ is the cofactor matrix of $D^2u^*$ and 
\begin{equation}\label{e211}
w^*=
\begin{cases}
\ [\det D^2u^*]^{-\alpha}, & \alpha>0 \
\text{and}\ \alpha\neq 1,\\[5pt]
\ -\log\det D^2u^*,& \alpha =0.
\end{cases}
\end{equation}

By a similar argument to that of Lemma \ref{l202}, we have the following result \cite{TW00, TW05, Z1}.
\begin{lemma}\label{l203}
Let $u^*$ be a smooth convex solution of \eqref{e210} in $\Omega^*$ in dimension 2, $u^*=0$ on $\partial\Omega^*$.
Then for any $y\in\Omega^*$, we have the a priori estimate
	\begin{equation}\label{e212}
		\det\,D^2u^*(y)\leq C,
	\end{equation}
where $C$ depends only on $\alpha, \dist(y,\partial\Omega^*), \sup_{\Omega^*}|u^*|, \sup_{\Omega^*}|Du^*|$, and $\inf f$.
\end{lemma}
By \eqref{e208} and \eqref{e212}, we have $\det\,D^2u\geq C$ has a positive lower bound. Note that the estimate depends on $\inf f$, but is independent of $\sup f$.

By Lemmas \ref{l202}, \ref{l203} and the Caffarelli-Guti\'errez theory \cite{CG}, we have the following H\"older and Sobolev space estimates.
\begin{theorem}\label{t201}
 Let $u\in C^4(\Omega)$ be a locally uniformly convex solution of \eqref{e205}.
\begin{itemize}
 \item[(i)] Assume $f\in L^\infty(\Omega)$. Then we have the estimate
  \begin{equation}\label{e213}
   \|u\|_{W^{4,p}(\Omega')}\leq C,
  \end{equation}
for any $p>1$ and $\Omega'\Subset\Omega$, where the constant $C$ depends on $n,p,\sup_\Omega|f|, \dist(\Omega',\partial\Omega)$, and the modulus of convexity of $u$.
 \item[(ii)] Assume $f\in C^\alpha(\Omega)$ for some $\alpha\in(0,1)$. Then
  \begin{equation}\label{e214}
   \|u\|_{C^{4,\alpha}(\Omega')}\leq C,
  \end{equation}
where $C$ depends on $n,\alpha,\|f\|_{C^\alpha(\Omega)}, \dist(\Omega',\partial\Omega)$, and the modulus of convexity of $u$.
\end{itemize}
\end{theorem}

Therefore, to prove the regularity of the maximizer $u_0$ in Lemma \ref{l201}, it suffices to prove, in view of Lemmas \ref{l202}, \ref{l203} and Theorem \ref{t201}, that (a) the maximizer $u_0$ can be approximated by smooth solutions to equation \eqref{e205} and (b) it is strictly convex. We will prove (a) and (b) in Sections 3 and 4, respectively.

\subsection{Rotations in $\mathbb{R}^{n+1}$}

In order to establish the estimate of the modulus of convexity, 
we need to treat convex functions as graphs in $\mathbb{R}^{n+1}$, and rotate the graphs in $\mathbb{R}^{n+1}$. 
When $\alpha=1/(n+2)$, the affine maximal surface equation \eqref{e104} is invariant under uni-modular transformations in $\mathbb{R}^{n+1}$. But this is not true for other $\alpha$. It has been proved in \cite{Z1} that for $\alpha=0$, under the rotations in $\mathbb{R}^{n+1}$, equation \eqref{e104} changes in a proper way such that the determinant estimate in Lemma \ref{l202} still holds.

For our purpose, we consider a more general functional 
	\begin{equation}\label{e215}
		J_\alpha(u)=A_\alpha(u)-\int_\Omega F(x,u)dx,
	\end{equation}
where $A_\alpha$ is in \eqref{e204}, $F(x, t)$ is a function on $\Omega\times \mathbb R$. 
Let $u$ be a locally critical point of the functional $J_\alpha$, thus it satisfies \eqref{e104} with the inhomogeneous term $f=F_t:=\frac{\partial F}{\partial t}$.

Consider the rotation $Z=TX$, given by $z_1=-x_{n+1}, z_{n+1}=x_1$, $z_i=x_i$ for $2\leq i\leq n$. Assume the graph of $u$, $\mathcal{G}_u=\{(x,u(x))\,:\,x\in\Omega\}$, can be represented by a convex function $z_{n+1}=v(z_1,\cdots,z_n)$ in $z$-coordinates over a domain $\hat\Omega$. Following the computation in \cite{Z2}, $v$ is a locally critical 
point of 
\begin{equation}\label{e216}
\hat J_\alpha(v)=\hat A_\alpha(v)-\int_{\hat\Omega}F(v,z_2,\cdots,z_n,-z_1),
\end{equation}
where
$$\hat A_\alpha(v)=
\begin{cases}
 \int_{\hat \Omega}[\det D^2v]^\alpha|v_1|^{1-(n+2)\alpha}\,dz, & \alpha>0,\\[5pt]
\int_{\hat \Omega} [\log\det D^2v 
-\frac{n+2}{2}\log(v_1^2)] (v_1^2)^{\frac{1}{2}}\,dz, &\ \alpha=0.
\end{cases}$$ 
When $\alpha>0$, by computing the Euler equation,
we can obtain the corresponding equation for $v$, that is,
	\begin{equation}\label{e217}
	\begin{split}
		\alpha v_1^{1-\alpha(n+2)}V^{ij}(d^{\alpha-1})_{ij}&+(1-\alpha)\alpha(n+2)(1-\alpha(n+2))v_1^{-\alpha(n+2)-1}v_{11}d^\alpha \\
			&+(1-\alpha(n+2))(2\alpha-2)v_1^{-\alpha(n+2)}(d^\alpha)_1= F_t,
	\end{split}
	\end{equation}
or equivalently, denoting $\lambda = 1-\alpha(n+2)$,
	\begin{equation}\label{e218}
		V^{ij}(d^{\alpha-1})_{ij}=g+\mathcal{F}_t,
	\end{equation}
where $(V^{ij})$ is the cofactor matrix of $(v_{ij})$, $d=\det\,D^2v$ and
	\begin{eqnarray*}
		g\!\!&=&\!\!2\lambda(1-\alpha)d^\alpha v^{ij}v_{ij1}\frac{1}{v_1}-(1-\alpha)(n+2)\lambda d^\alpha\frac{v_{11}}{v_1^2}, \\
		\mathcal{F}_t\!\!&=&\!\!\alpha^{-1} \frac{F_t}{v_1^\lambda}, \quad F_t = \frac{\partial F}{\partial t}(v,z_2,\cdots,z_n,-z_1).
	\end{eqnarray*}
When $\alpha=0$, by a similar computation we obtain \eqref{e218} with 
$\mathcal{F}_t={F_t}/{v_1}$.

\subsection{A priori estimates}

In this subsection, we obtain the a priori determinant estimates 
under the rotation transform $Z=TX$. 
Let $v$ be a smooth solution of \eqref{e218} satisfying 
	\begin{equation}\label{e219}
		v\geq0,\ \ v\geq z_1,\ \ v_1\geq0,
	\end{equation}
and $v(0)$ is as small as we want such that for the positive constant $s$ and $h$ in $(0,1/2)$, $\hat\Omega_{s,h}$ is a nonempty open set, where
	\begin{equation}\label{e220}
		\hat\Omega_{s,h}=\{z\,:\,v(z)<sz_1+h\}.
	\end{equation}
Set $\hat v:=v-sz_1-h$, then $\hat\Omega_{s,h}=\{z\,:\,\hat v(z)<0\}$ and $\hat v$ satisfies
	\begin{equation}\label{e221}
		\hat V^{ij}(\hat d^{\alpha-1})_{ij}=\hat g+\mathcal{\hat F}_t,
	\end{equation}
where $(\hat V^{ij})$ is the cofactor matrix of $(\hat v_{ij})$, $\hat d=\det\,D^2\hat v$ and
	\begin{eqnarray*}
		\hat g\!\!&=&\!\!2\lambda(1-\alpha)\hat d^\alpha \frac{\hat v^{ij}\hat v_{ij1}}{\hat v_1+s}-(1-\alpha)(n+2)\lambda \hat d^\alpha\frac{\hat v_{11}}{(\hat v_1+s)^2}, \\
		\mathcal{\hat F}_t\!\!&=&\!\!\alpha^{-1} \frac{\hat F_t}{(\hat v_1+s)^\lambda}, \quad \hat F_t = \frac{\partial F}{\partial t}(\hat v+sz_1+h,z_2,\cdots,z_n,-z_1).
	\end{eqnarray*}

\begin{lemma}\label{l204}
Assume $0\leq\alpha\leq\frac{1}{n+2}$. Let $\hat v$ be a smooth solution of \eqref{e221} in $\hat\Omega_{s,h}$ and $\hat v=0$ on $\partial\hat\Omega_{s,h}$.
Then for any $z\in\hat\Omega_{s,h}$, we have the a priori estimate
	\begin{equation}\label{e222}
		\det\,D^2\hat v\leq C,
	\end{equation}
where $C$ depends only $n,\alpha,\dist(z,\partial\hat\Omega_{s,h}),\sup_{\hat\Omega_{s,h}}|\hat v|, \sup_{\hat\Omega_{s,h}}|D\hat v|$ and $\sup{\hat F}_t$.
\end{lemma}

\begin{proof}
When $\alpha=\frac{1}{n+2}$, the estimate \eqref{e222} easily follows from the affine invariant property. Note that in this case, $\lambda=0$ and $\hat g$ in \eqref{e221} vanishes.  
The case of $\alpha=0$ was contained in \cite{Z1}. Here we give a proof for the remaining case
$0<\alpha< \frac{1}{n+2}$ as follows.
Let 
	\begin{equation}\label{e223}
		\eta=\log w-\beta\log(-\hat v)-A|D\hat v|^2,
	\end{equation}
where $w=\hat d^{\alpha-1}$, and $\beta, A$ are positive constants to be determined later. 
Since $\eta\to+\infty$ on $\partial\hat\Omega_{s,h}$, it attains a minimum at some point $z_0\in\hat\Omega_{s,h}$.
At $z_0$, we then have
	\begin{eqnarray}
		0\!\!&=&\!\!\eta_i=\frac{w_i}{w}-\frac{\eta\hat v_i}{\hat v}-2A\hat v_k\hat v_{ki}, \label{e224}\\
		0\!\!&\leq&\!\![\eta_{ij}]=\left[\frac{w_{ij}}{w}-\frac{w_iw_j}{w^2}-\frac{\beta\hat v_{ij}}{\hat v}+\frac{\beta\hat v_i\hat v_j}{\hat v^2}-2A\hat v_{ki}\hat v_{kj}-2A\hat v_k\hat v_{kij}\right] \label{e225}
	\end{eqnarray}
as a matrix. Since $w=[\det\,D^2\hat v]^{\alpha-1}$, we have
	\begin{equation}\label{e226}
		\hat v^{ij}\hat v_{kij}=(\log\det\,D^2\hat v)_k=\frac{1}{\alpha-1}\frac{w_k}{w},
	\end{equation}
where $(\hat v^{ij})=\hat d^{-1}(V^{ij})$ is the inverse of $D^2\hat v$.
We may assume that $\hat d>1$, otherwise the proof is done. 
Hence,
  \begin{equation}\label{e227}
    \frac{\hat v^{ij}w_{ij}}{w} = \frac{\hat g+\mathcal{\hat F}_t}{\hat d^\alpha} \leq -2\lambda\frac{w_1}{w}(\hat v_1+s)^{-1}-(1-\alpha)(n+2)\lambda\frac{\hat v_{11}}{(\hat v_1+s)^2}+\frac{\sup\hat F_t}{\alpha(\hat v_1+s)^\lambda}.
  \end{equation}
Therefore, we obtain
  \begin{eqnarray}\label{e228}
   0 \!\!&\leq&\!\! \hat v^{ij}\eta_{ij} \nonumber\\
	 &\leq&\!\! \frac{\sup\hat F_t}{\alpha(\hat v_1+s)^\lambda}+\frac{\lambda(\alpha-1)(n+2)\hat v_{11}}{(\hat v_1+s)^2}-4A\lambda\sum_{k=1}^n\frac{\hat v_{1k}\hat v_k}{\hat v_1+s}-2\lambda\beta\frac{\hat v_1}{(\hat v_1+s)\hat v} \nonumber\\
	 &&	-\frac{\beta n}{\hat v}-\left(2A\bigtriangleup\hat v-\frac{4A^2\alpha}{1-\alpha}\hat v_{ij}\hat v_i\hat v_j\right)-\left(4A\beta-\frac{2A\beta}{1-\alpha}\right)\frac{|D\hat v|^2}{\hat v}-(\beta^2-\beta)\frac{\hat v^{ij}\hat v_i\hat v_j}{\hat v^2} \\
	 &\leq&\!\! \frac{\sup\hat F_t}{\alpha(\hat v_1+s)^\lambda}-2\lambda\beta\frac{\hat v_1}{(\hat v_1+s)\hat v}-\frac{\beta n}{\hat v}-\frac{A}{2}\bigtriangleup\hat v-\left(4A\beta-\frac{2A\beta}{1-\alpha}\right)\frac{|D\hat v|^2}{\hat v}, \nonumber
  \end{eqnarray}
with the choice of $\beta>1$ and $A$ small enough such that
  \begin{equation}\label{e229}
   \frac{A}{2}\bigtriangleup\hat v \geq \frac{4A^2\alpha}{1-\alpha}\hat v_{ij}\hat v_i\hat v_j+CA^2\hat v_{11},
  \end{equation}
where $C$ is a constant depending only on $n, \alpha$ and $|D\hat v|$. Observing that
  \begin{equation}\label{e230}
   \frac{\hat v_1}{(\hat v_1+s)\hat v}=\frac{1}{\hat v}-\frac{s}{(\hat v_1+s)\hat v},
  \end{equation}
by choosing $\beta$ large enough such that
  \begin{equation}\label{e231}
   (-\hat v)(\hat v_1+s)^{1-\lambda}\sup\hat F_t \leq 2s\alpha\lambda\beta,
  \end{equation}
we have
  \begin{equation}\label{e232}
   -\frac{\beta(n+2\lambda)}{\hat v}-\frac{A}{2}\bigtriangleup\hat v-\left(4A\beta-\frac{2A\beta}{1-\alpha}\right)\frac{|D\hat v|^2}{\hat v}\geq0,
  \end{equation}
which implies
  \begin{equation}\label{e233}
   (-\hat v)\bigtriangleup\hat v\leq C.
  \end{equation}
It follows that $\eta(z)\geq\eta(z_0)\geq-C$ and so \eqref{e222} holds.
\end{proof}

\section{Approximations}

Let $u_0$ be the maximizer of \eqref{e102}. In this section, we prove that $u_0$ can be approximated by a sequence of smooth solutions to equation \eqref{e205}. The approximation enables us to apply the a priori estimates in Section 2.
For Monge-Amp\`ere equations, or general second order equations, one can
obtain the approximation from a perturbation of the equation. However, 
the perturbation does not work for fourth order
equations because of the lack of maximum principle. 
We will construct the approximation using a penalty method to the functionals. 
We also need to deal with the difficulty coming from the obstacle.

\subsection{Obstacle approximation}

Let $u_0$ be the maximizer of $J_\alpha$ in $\mathcal{S}[\varphi,\psi]$. We construct a sequence
of penalized functionals whose maximizers do not contact the obstacle 
and approximate $u_0$. Let $\mathcal{S}[\varphi, u_0]$ be the set of convex functions with 
$u_0$ as the obstacle, namely,
  \begin{equation}\label{e301}
   \mathcal{S}[\varphi,u_0]=\left\{v\in C(\overline\Omega)\,:\,v \mbox{ convex }, v|_{\partial\Omega}=\varphi, Dv(\Omega)\subset D\varphi(\overline\Omega), v\geq u_0\mbox{ in }\Omega\right\},
  \end{equation}
where $\varphi$ is a smooth, uniformly convex function defined on a neighborhood of $\overline\Omega$.

\begin{lemma}\label{l301}
   Suppose $0\leq \alpha\leq\frac{1}{n+2}$. There exists a sequence of functions $\{u_i\}$ in $\mathcal{S}[\varphi, u_0]$ such that 
   each $u_i$ is the maximizer of the functional 
   $$J^i_{\alpha}(v)=J_\alpha(v)-\int_\Omega G_i(x,v), \ v\in \mathcal{S}[\varphi, u_0]$$
   and  $u_i\to u_0$ as $i\to \infty$, 
  where $G_i(x,t)$ is a smooth, convex function monotone decreasing in $t$ . 
  Furthermore, there is no obstacle for $u_i$ in $\Omega$, 
  i.e., $u_i(x)>u_0(x)$, $x\in \Omega$.
\end{lemma}

\begin{proof}
First, we consider a penalized problem. The idea is inspired by \cite{STW}. Define
\begin{equation}\label{e302}
J_{\alpha, g}(v)=J_\alpha(v)-\int_\Omega G(x,v),
\end{equation}
where $G(x,t)$ is a smooth, convex function monotone decreasing in $t$ such that
  \begin{equation}\label{e303}
     G(x,t)\geq a(x)(t-u_0(x))^{-n}\quad\mbox{for }t>u_0(x), \ x \in \Omega.
  \end{equation}
Here $a$ is  a positive function in $\Omega$, with $a(x)\to 0$ fast enough as $x\to\partial\Omega$ such that the set 
$\{v\in \mathcal{S}[\varphi, u_0]\,:\,J_{\alpha, g}(v)>-\infty\}\neq\emptyset$. 
It is clear that $J_{\alpha,g}$ is still concave, upper semi-continuous and bounded from above.    Hence there is a unique maximizer $v_g$ to the problem
  \begin{equation}\label{e304}
   \sup\{J_{\alpha, g}(v)\,:\,v\in \mathcal{S}[\varphi, u_0]\}.
  \end{equation}
  
We claim that for any $x\in\Omega$,
	\begin{equation}\label{e305}
		v_g(x)>u_0(x). 
	\end{equation}
Indeed, if there is a point $x_0\in\Omega$ such that $v_g(x_0)=u_0(x_0)$, by convexity the graphs of $v_g$ and $u_0$ are bounded by the cone $\mathcal{K}$ and the hyperplane $\mathcal{P}$, where $\mathcal{K}$ has the vertex at $(x_0,u_0(x_0))$ and passes through $(\partial\Omega,u_0|_{\partial\Omega})$, and $\mathcal{P}$ is the support plane of $u_0$ at $x_0$. Then we have $|v_g(x)-u_0(x)|\leq C|x-x_0|$. Hence by the assumption on $G(x, t)$,
	\begin{equation}\label{e306}
		\int_{\Omega}G(x,v_g(x))\geq C\int_{\Omega}|x-x_0|^{-n}=\infty.
	\end{equation}
That is, $v_g$ cannot be a maximizer. 

Replacing $G$ by $\varepsilon_i G$ for a sequence $\varepsilon_i\to0$, accordingly there exists a sequence of maximizers $v_{\varepsilon_i}$ to \eqref{e304}. Since $u_0$ is itself a maximizer, we have $v_{\varepsilon_i}\to u_0$ as $\varepsilon_i\to 0$ by the concavity of the functional $J_\alpha$. Hence, the sequence $u_i$ can be chosen from $v_{\varepsilon_i}$.
\end{proof}

\begin{remark}
If $u_i$ is smooth, it satisfies the equation
	\begin{equation}\label{e307}
		L[u]=f+g_i \quad\mbox{in }\Omega,
	\end{equation}
where $L$ is the operator in \eqref{e104}, 
and $g_i=\frac{\partial}{\partial t}G_i(x,t)$ at $t=u_i(x)$. 
In the later proof of strict convexity, we will need the 
upper bound estimate for the determinant of $D^2u_0$ which depends on
$\sup f$. Since $g_i<0$ in the above approximation, the estimate in Section 2 still applies 
when turning to the sequence $u_i$.

When studying the strict convexity of enclosed convex hypersurfaces with maximal affine area, one can assume $u_0$ is equal to a linear function $\ell$ on $\partial\Omega$ \cite{STW}, then the above proof can be simplified. 
\end{remark}

\begin{remark}
In fact, the approximation in Lemma \ref{l301} applies on any subdomain $\Omega'\subset\Omega$.
Instead of considering the boundary $\varphi$, one can consider
	\[\mathcal{S}_{\Omega'}[u_0]=\{v\in C(\overline\Omega')\,:\,v \mbox{ convex}, v|_{\partial\Omega'}=u_0|_{\partial\Omega'}, Dv(\Omega')\subset Du_0(\overline\Omega'), v\geq u_0\},\]
and then obtain a local approximation sequence. 
\end{remark}

\subsection{Smooth approximation}

Let $u$ be the maximizer of \eqref{e304}. From the obstacle approximation, $u$ is also the maximizer of \eqref{e302} over the set
	\begin{equation}\label{e308}
		\mathcal{S}[\varphi,\Omega]=\left\{v\in C(\overline\Omega)\,:\,v \mbox{ convex}, v|_{\partial\Omega}=\varphi|_{\partial\Omega}, Dv(\Omega)\subset D\varphi(\overline\Omega)\right\}.
	\end{equation}
In this subsection, we prove that $u$ can be approximated by smooth solutions of 
  \begin{equation}\label{e309}
   U^{ij}w_{ij}=f(x,u),
  \end{equation}
where $U^{ij}$ is the cofactor of $D^2u$ and $w=[\det\,D^2u]^{\alpha-1}$.
This approximation enables us to apply the a priori estimates in Section 2.

\begin{lemma}\label{l302}
Let $u$ be the maximizer of \eqref{e304}. Suppose $\partial\Omega$ is Lipschitz continuous. 
Then there exists a sequence of smooth solutions to equation \eqref{e309} converging locally uniformly to the maximizer $u$.
\end{lemma}

To prove the approximation, first we recall the existence and regularity of solutions of the following second boundary value problem \cite{TW05}. 
Let $B=B_R(0)$ be a ball such that $\Omega\Subset B_{R-1}(0)$ and $\phi$ is a smooth, uniformly convex function in $B$ and $\phi=c^*$ is constant on $\partial B$. 
Let 
  \begin{equation}\label{e310}
   H(t)=(1-t^2)^{-2n}
  \end{equation}
be a nonnegative smooth function in the interval $(-1,1)$. When $|t|>1$, we can formally define $H(t)=+\infty$. 
Extend the function $f$ in \eqref{e309} to $B$ such that
  \begin{equation}\label{e311}
   f=\left\{\begin{array}{ll}
      f(x) & x\in\Omega,\\
      H'(u-\phi(x)) & x\in B\setminus\Omega.
   \end{array}\right.
  \end{equation}
  
\begin{lemma}\label{l303}
 Suppose $\partial\Omega$ is Lipschitz continuous. Then there is a uniformly convex solution $u\in W^{4,p}_{loc}(B)\cap C^{0,1}(\overline B)$ (for all $p<\infty$) with $\det\,D^2u\in C^0(\overline\Omega)$ of the boundary value problem
  \begin{eqnarray}\label{e312}
   U^{ij}w_{ij}\!\!&=&\!\!f(x,u)\quad\mbox{in }B,\\
   u\!\!&=&\!\!\phi \, (=c^*)\quad\mbox{on }\partial B,\nonumber \\
   w\!\!&=&\!\!1\quad\mbox{on }\partial B. \nonumber
  \end{eqnarray}
\end{lemma}

The existence and regularity of solutions of \eqref{e312} was previously obtained in \cite{TW05,TW10} for $\alpha=\frac{1}{n+2}$, and \cite{Z1} for $\alpha=0$. The crucial ingredient is to establish  
  \begin{equation}\label{e313}
   |f(x,u)|\leq C
  \end{equation}
for some constant $C>0$ independent of $u$. Once $f$ is bounded, the regularity and existence of solutions follow easily from \cite{TW05}. The global $C^{4,\alpha}$ regularity was recently proved in \cite{TW08}.
Following the argument in \cite{TW05}, one can easily check the proof works for all $\alpha\in(0,\frac{1}{n+2})$. 

Now, we show that the maximizer of $J_\alpha(u)$ can be approximated by smooth solutions to equation \eqref{e309}.
\begin{proof}[Proof of Lemma \ref{l302}]
By assumption $\varphi$ is smooth, uniformly convex in a neighborhood of $\Omega$, so we can extend it to $B=B_R$ such that $\varphi$ is convex in $B$, $\varphi\in C^{0,1}(\overline B)$ and $\varphi$ is constant on $\partial B$. Replacing $\varphi$ by $\varphi+\left(|x|-R+\frac12\right)_+^2$, where
	\[\left(|x|-R+\frac12\right)_+=\max\left\{|x|-R+\frac12,0\right\},\]
we also assume that $\varphi$ is uniformly convex in $\{x\in\mathbb{R}^n\,:\,R-\frac12<|x|<R\}$.
Consider the second boundary value problem \eqref{e312} with
	\begin{equation}\label{e314}
		f_j(x,u)=\left\{\begin{array}{ll}
			f & \mbox{in }\Omega \\
			H'_j(u-\varphi) & \mbox{in }B\setminus\Omega,
		\end{array}\right.
	\end{equation}
where $H_j(t)=H(4^jt)$ and $H$ is defined by \eqref{e310}.
By Lemma \ref{l303} there is a solution $u_j$ satisfying 
	\begin{equation}\label{e315}
		|u_j-\varphi|\leq 4^{-j},\quad x\in B\setminus\Omega.
	\end{equation}
By the convexity, $u_j$ sub-converges to a convex function $\bar u$ in $B$ as $j\to\infty$. Note that $\bar u=\varphi$ in $B\setminus\Omega$. Hence, $\bar u\in\mathcal{S}[\varphi,\Omega]$ when restricted in $\Omega$. Using a similar argument as in \cite{TW10} and \cite{Z1}, one can show that $\bar u$ is the maximizer of \eqref{e302} over the set \eqref{e308}. 
By the uniqueness of maximizer, we obtain $\bar u=u$.
The main ingredients of the argument in \cite{TW10} are the upper semicontinuity and the concavity of the functional \eqref{e302}, which hold for all $\alpha\in[0,\frac{1}{n+2}]$, see Lemma \ref{l201}.
\end{proof}

\section{Strict convexity}

In this section, we prove the strict convexity of $u_0$ in dimension two.
Let $\mathcal{G}_0$ be the graph of $u_0$. If $u_0$ is not strictly convex, then $\mathcal{G}_0$ contains a line segment.
Let $\ell(x)$ be a tangent function of $u_0$ at the segment and denote by
  \begin{equation}\label{e401}
   \mathcal{C}=\{x\in\Omega\,:\,u_0(x)=\ell(x)\}
  \end{equation}
the contact set. The set $\mathcal{C}\subset\mathbb{R}^2$ is bounded and convex.

We say a point $x_0\in\partial U$ is an extreme point of a bounded convex domain $U\subset\mathbb{R}^n$ if there is a hyperplane $P$ such that $\{x_0\}=P\cap\partial U$, namely the intersection $P\cap\partial U$ is the single point $x_0$. We divide our discussion into the following two cases:
\begin{itemize}
 \item[(a)]: $\mathcal{C}$ has an exteme point $x_0$, which is an interior point of $\Omega$;
 \item[(b)]: All extreme points of $\mathcal{C}$ lie on $\partial\Omega$.
\end{itemize}

We will rule out the possibility of both cases, and thus $u_0$ is strictly convex. The basic observation is that a convex function with a bounded Monge-Amp\`ere measure is differentiable at any point on its graph, not lying on a line segment joining two boundary points, \cite{Caf1}. In dimension two, recall the following 
\begin{lemma}[\cite{TW05}]\label{l401}
 Suppose $u$ is a nonnegative convex function in a domain $\Omega\subset\mathbb{R}^2$. The origin $0\in\Omega$ is an interior point. $u$ satisfies $u>0$ on $\partial\Omega$, $u(0)=0$ and $u(x_1,0)\geq|x_1|$. Then the Monge-Amp\`ere measure $\mu[u]$ cannot be a bounded function. 
\end{lemma}

\subsection{Strict convexity I}

First we rule out the possibility that $\mathcal{G}_0$ contains a line segment with one endpoint in the interior of $\Omega$.

\begin{lemma}\label{l402}
 $\mathcal{C}$ contains no extreme points in the interior of $\Omega$. 
\end{lemma}

\begin{proof}
The proof is by contradiction arguments as in \cite{TW00,Z1}.  
Without loss of generality, we may assume that $\ell(x)=0$, the origin is an extreme point of $\mathcal{C}$ and the segment $\{(x_1,0)\,:\,0\leq x_1\leq 1\}\subset\mathcal{C}$. From the approximation argument, we can choose a sequence of functions $\{u_k\}$ converging to $u_0$ such that $u_k$ is a solution of \eqref{e307}. 
Let $\mathcal{G}_k$ be the graph of $u_k$. Then $\mathcal{G}_k$ converges in the Hausdorff distance to $\mathcal{G}_0$.

For $\varepsilon>0$ small enough, let
  \begin{equation}\label{e402}
   \ell_\varepsilon=-\varepsilon x_1+\varepsilon,\quad\mbox{and } \Omega_\varepsilon=\{u<\ell_\varepsilon\}.
  \end{equation}
Let $T_\varepsilon$ be a coordinates transformation that normalizes the domain $\Omega_\varepsilon$. 
Define
  \begin{equation}\label{e403}
   u_\varepsilon(x)=\frac{1}{\varepsilon}u(T_\varepsilon^{-1}(x)),\quad u_{k,\varepsilon}=\frac{1}{\varepsilon}u_k(T_\varepsilon^{-1}(x)),\quad x\in\tilde\Omega_\varepsilon,
  \end{equation}
where $\tilde\Omega_\varepsilon=T_\varepsilon(\Omega_\varepsilon)$ is normalized. After this transformation we have the following observations:

(i) The equation $U^{ij}w_{ij}=f$ with $w=[\det\,D^2u]^{\alpha-1}$, $0\leq\alpha\leq\frac14$, will become
  \begin{equation}\label{e404}
   \tilde U^{ij}\tilde w_{ij}=\tilde f,
  \end{equation}
where $\tilde U^{ij}$ is the cofactor of $D^2\tilde u$, 
	\[\tilde w=[\det\,D^2\tilde u]^{\alpha-1},\quad \mbox{and } \tilde f=|T_\varepsilon|^{-2\alpha}\varepsilon^{1-2\alpha}f.\] 
In fact, since $T_\varepsilon$ normalizes $\Omega_\varepsilon$, $|T_\varepsilon|^{-1}\leq|\Omega_\varepsilon|\leq C$. Therefore, $\tilde f\to0$ as $\varepsilon\to0$.

(ii) Denote by $\mathcal{G}_\varepsilon$ and $\mathcal{G}_{k,\varepsilon}$ the graphs of $u_\varepsilon$ and $u_{k,\varepsilon}$, respectively. 
Taking $k\to\infty$, it is clear that $u_{k,\varepsilon}\to u_\varepsilon$ and $\mathcal{G}_{k,\varepsilon}$ converges in the Hausdorff distance to $\mathcal{G}_\varepsilon$. 
Then taking $\varepsilon\to0$, we have that the domain $\tilde\Omega_\varepsilon$ sub-converges to a normalized domain $\tilde\Omega$ and $u_\varepsilon$ sub-converges to a convex function $\tilde u$ defined in $\tilde\Omega$. 
We also have $\mathcal{G}_\varepsilon$ sub-converges in the Hausdorff distance to a convex surface $\mathcal{\tilde G}_0\in\mathbb{R}^3$. 

(iii) By a rotation of coordinates, the convex surface $\mathcal{\tilde G}_0$ satisfies
	\begin{equation}\label{e405}
		\mathcal{\tilde G}_0\subset\{y_1\geq0\}\cap\{y_3\geq0\}
	\end{equation}
and $\mathcal{\tilde G}_0$ contains two segments
	\begin{equation}\label{e406}
		\{(0,0,y_3)\,:\,0\leq y_3\leq 3\},\quad \{(y_1,0,0)\,:\,0\leq y_1\leq 1\}.
	\end{equation}

Hence, by (i)--(iii) we can assume that there is a sequence of solutions $\tilde u_k$ of
	\begin{equation}\label{e407}
		U^{ij}w_{ij}=\varepsilon_k f\quad\mbox{in }\tilde\Omega_k,
	\end{equation}
where $w=[\det\,D^2u]^{\alpha-1}$, and $\varepsilon_k\to0$ such that the normalized domain $\tilde\Omega_k$ converges to $\tilde\Omega$, $\tilde u_k$ converges to $\tilde u$ and the graph of $\tilde u_k$, denoted by $\mathcal{\tilde G}_k$ converges in the Hausdorff distance to $\mathcal{\tilde G}_0$. 

Note that in $y$-coordinates, $\mathcal{\tilde G}_0$ is not a graph of a function near the origin. By adding some linear function to $\tilde u_k$ and $\tilde u$ and making a rotation of coordinates in $\mathbb{R}^3$, i.e., $z_i=R_{ij}y_j$, where $(R_{ij})$ is a $3\times3$ rotation matrix, $\mathcal{\tilde G}_k, \mathcal{\tilde G}_0$ can be represented by $z_3=v_k(z_1,z_2), z_3=v(z_1,z_2)$, respectively \cite{Z1}. 
Moreover, $v_k$ is a solution of the equation given in \S2.4 near the origin, $v$ satisfies
	\begin{equation}\label{e408}
		v\geq\frac12|z_1|,\quad\mbox{and}\quad v(z_1,0)=\frac12|z_1|.
	\end{equation}

As we know that $\mathcal{\tilde G}_k$ converges in the Hausdorff distance to $\mathcal{\tilde G}_0$, in the new coordinates, $v_k$ converges locally uniformly to $v$. Let $\mathcal{\tilde C}=\{(z_1,z_2),\:\,v(z_1,z_2)=0\}$, and
	\begin{equation}\label{e409}
		L=\{(z_1,z_2,0)\,:\,(z_1,z_2)\in\mathcal{\tilde C}\}
	\end{equation}
in $z$-coordinates. $L$ could be a single point (Case I) or a segment on $z_2$-axis (Case II). 

\emph{Case I}:
In this case, $v$ is strictly convex at $(0,0)$. The strict convexity implies that $Dv$ is bounded on the sub-level set $S_{h,v}(0)$ for small $h>0$. Hence, by locally uniform convergence, $Dv_k$ are uniformly bounded on $S_{h/2,v_k}(0)$. By Lemma \ref{l204}, we have the determinant estimate
	\begin{equation}\label{e410}
		\det\,D^2v_k\leq C
	\end{equation}
near the origin, where the constant $C$ is uniform with respect to $k$. By the weak continuity of Monge-Amp\`ere measure, $\mu[v]\leq C$ near the origin. The contradiction
follows by Lemma 4.1.

\emph{Case II}:
In this case, $L$ is a segment, we may also assume that $0$ is an end point of $L$, i.e.,
	\[\mathcal{\tilde C}=\{(0, z_2)\,:\,-1\leq z_2\leq 0\}.\]
Define the linear function
	\begin{equation}\label{e411}
		\ell_\varepsilon(z)=\delta_\varepsilon z_2+\varepsilon
	\end{equation}
and $\omega_\varepsilon=\{z\,:\,v(z)\leq\ell_\varepsilon\}$, where $\delta_\varepsilon, \varepsilon$ are chosen such that $\varepsilon\delta_\varepsilon^{-1}\to0$ as $\varepsilon\to0$. 
	By taking the similar transformations and normalizations as in \eqref{e402}, \eqref{e403} with respect to $z_2$ direction, one can reduce Case II to Case I. The proof is then finished.
	
\end{proof}

\subsection{Strict convexity II}

Next, we rule out the possibility of case (b) that all extreme points of $\mathcal{C}$ lie on the boundary $\partial\Omega$. Recall the definition of $\mathcal{C}$ in \eqref{e401}, and define the set $T:=\{x\in\Omega\,:\,u(x)=\psi(x)\}$, where $\psi$ is the obstacle. 

\begin{lemma}\label{l403}
 Let $u_0\in\mathcal{S}[\varphi,\psi]$ be the maximizer. The obstacle $\psi$ is a convex function in $\Omega$ satisfying $\psi<\varphi$ on $\partial\Omega$. If all extreme points of $\mathcal{C}$ lie on the boundary $\partial\Omega$, then $\dist(\mathcal{\overline C},\overline{T})>c_0$ for some positive constant $c_0$. 
\end{lemma}

\begin{proof}
This follows easily from the convexity.
\end{proof}

\begin{lemma}\label{l404}
 Assume that $\varphi$ is uniformly convex in a neighborhood of $\Omega$. then $\mathcal{G}_0$ contains no line segments with both endpoints on $\partial\mathcal{G}_0$.
\end{lemma}

\begin{proof}
By Lemma \ref{l403}, we can restrict our discussion on a sub-domain $\Omega'\subset\Omega$ satisfying $\dist(\Omega',T)>c_0$ and $\{\mbox{extreme points of }\mathcal{C}\}\subset\partial\Omega'\cap\partial\Omega$. 
Let $u_0$ be the maximizer of $J_\alpha$ and
  \begin{equation}\label{e412}
   \bar S[u_0,\Omega']:=\{v\in C(\overline\Omega')\,:\,v \mbox{ convex }, v_{\partial\Omega'}=u_0, N_v(\Omega')\subset N_{u_0}(\overline\Omega')\}.
  \end{equation}
Note that since $\dist(\Omega',T)>c_0$, when restricting on $\Omega'$, $u_0$ is naturally a maximizer of $J_\alpha$ over $\bar S[u_0,\Omega']$ without obstacle. 
Therefore, we can apply a similar local approximation in \cite{TW10} as follows:
\begin{quote}
\emph{Claim:} There exists a sequence of smooth, uniformly convex solutions $u_m\in W^{4,p}(\Omega')$ $(\forall p<\infty)$ of
  \begin{equation}\label{e413}
   U^{ij}w_{ij}=f+\beta_m\chi_{D_m}\quad\mbox{in }\Omega'
  \end{equation}
such that
  \begin{equation}\label{e414}
   |u_m-u|\to0\quad\mbox{uniformly in }\Omega',
  \end{equation}
where $D_m=\{x\in\Omega'\,:\,\dist(x,\partial\Omega')<2^{-m}\}$, $\chi$ is the characteristic function, and $\beta_m>0$ is a constant.
Furthermore, we can choose $\beta_m$ sufficiently large ($\beta_m\to\infty$ as $m\to\infty$) such that for any compact, proper subset $K\subset N_{u_0}(\Omega')$,
  \begin{equation}\label{e415}
   K\subset N_{u_m}(\Omega')
  \end{equation}
provided $m$ is sufficently large, where $N_u$ is the normal mapping introduced in Section 2. 
\end{quote}
The proof of the claim is contained in \cite{TW10} for the case $\alpha=\frac{1}{n+2}$, see also \cite{Z1} for the case $\alpha=0$. The idea is similar to the proof of Lemma \ref{l302}. But instead of considering the second boundary value problem with inhomogeneous term \eqref{e314}, we consider a weighted one
	\begin{equation}\label{e416}
		f_{m,j}=\left\{\begin{array}{ll}
			f+\beta_m\chi_{D_m} & \mbox{in }\Omega' \\
			H_j'(u-u_0) & \mbox{in }B_R\setminus\Omega'
		\end{array}\right.
	\end{equation}
where $H_j(t)=H(4^jt)$ given by \eqref{e310}, $B_R$ is a large ball enclosing $\Omega'$. By Lemma \ref{l303}, there is a solution $u_{m,j}$ satisfying
	\begin{equation}\label{e417}
		|u_{m,j}-u_0|\leq 4^{-j},\quad x\in B_R\setminus\Omega'.
	\end{equation}
By the convexity, $u_{m,j}$ sub-converges to a convex function $u_m$ as $j\to\infty$ and $u_m=u_0$ in $B_R\setminus\Omega'$. Note that $u_m\in\mathcal{S}[u_0,\Omega']$ when restricted in $\Omega'$, therefore, $u_m$ converges to a convex function $u_\infty$ in $\mathcal{S}[u_0,\Omega']$ as $m\to\infty$. Similarly, one can show that $u_\infty$ is the maximizer of $J_\alpha$ over the set $\mathcal{S}[u_0,\Omega']$. By the uniqueness of maximizer, we have $u_\infty=u_0$ and obtain the claim. See \cite{TW10,Z1} for more details.

Now, suppose that $\ell$ is a line segment in $\mathcal{G}_0$ with both end points on $\partial\mathcal{G}_0$. By substracting a linear function, we assume that $u_0\geq0$ and $\ell$ lies in $\{x_3=0\}$. From the definition of $\Omega'$, we also have $\ell\subset\Omega'$ with both end points on $\partial\Omega'\cap\partial\Omega$. By a traslation and a dilation of the coordiantes, we may assume furthermore that
  \begin{equation}\label{e418}
   \ell=\{(0,x_2,0)\,:\,-1\leq x_2\leq 1\}
  \end{equation}
with the endpoints $(0,\pm1)\in\partial\Omega'\cap\partial\Omega$.

Since $\varphi$ is smooth, uniformly convex in a neighborhood of $\Omega$ and $u_0=\varphi$ on $\partial\Omega$, it follows 
  \begin{equation}\label{e419}
   u_0(x)=\varphi(x)\leq\frac{C}{2}|x_1|^2,\quad x\in\partial\Omega'\cap\partial\Omega.
  \end{equation}
By the convexity of $u_0$,	
  \begin{equation}\label{e420}
   u_0(x)\leq\frac{C}{2}|x_1|^2,\quad x\in\Omega'.
  \end{equation}

Consider the Legendre transform $u_0^*$ of $u_0$ in $\Omega^*=D\varphi(\Omega)$, given by
  \begin{equation}\label{e421}
   u_0^*(y)=\sup\{x\cdot y-u_0(x),\ \ x\in\Omega\},\quad y\in\Omega^*.
  \end{equation}
Since both endpoints $(0,\pm1)\in\partial\Omega'\cap\partial\Omega$, by the uniform convexity of $\varphi$, $0\notin D\varphi(\partial\Omega)$. Hence $0\in\Omega^*$ is an interior point. By \eqref{e419}, \eqref{e420} we have
  \begin{eqnarray}
   u_0^*(0,y_2)\!\!&\geq&\!\!|y_2|, \label{e422} \\
   u_0^*(y)\!\!&\geq&\!\!\frac{1}{2C}y_1^2. \label{e423}
  \end{eqnarray}
Therefore, $\det\,D^2u_0^*$ is not bounded from above near the origin by Lemma \ref{l401}. 

But on the other hand, by the a priori estmiate in Lemma \ref{l203}, $\det\,D^2u_0^*$ must be bounded. Indeed, consider the Legendre transform $u_m^*$ of $u_m$. By the approximations \eqref{e414}, \eqref{e415}, and \eqref{e210}, $u_m^*$ satisfies the equation
  \begin{equation}\label{e424}
   U^{*ij}w_{ij}^*=-f_m(Du^*)\det\,D^2u^*\quad\mbox{in }\Omega^*_{\varepsilon_m},
  \end{equation}
where $f_m=f+\beta_m\chi_{D_m}$ and
  \[\Omega^*_{\varepsilon_m}=\{y\in\Omega^*\,:\,\dist(y,\partial\Omega^*)>\varepsilon_m\}\]
with $\varepsilon_m\to 0$ as $m\to\infty$. By the growth estimates \eqref{e422} and \eqref{e423}, $u_0^*$ is strictly convex at $0$, the set $\{u_0^*<h\}$ is strictly contained in $\Omega^*$ provided $h>0$ is small. Note that $u_m^*$ converges to $u_0^*$. By Lemma \ref{l203} we have the estimate
  \[\det\,D^2u_m^*\leq C_1\]
near the origin in $\Omega^*$. Note also that in Lemma \ref{l203}, the constant $C_1$ depends on $\inf f$ but not on $\sup f$. In other words, the large constant $\beta_m$ in \eqref{e413} does not affect the bound $C_1$. Therefore, sending $m\to\infty$,  we obtained	
  \[\det\,D^2u_0^*\leq C\]
near the origin. This is in contradiction with the assertion that $\det\,D^2u_0^*$ is not bounded from above near the origin. 
\end{proof}

\section{Regularity}

We can now give the proof of Theorem \ref{t101}, which is divided into two parts:

\subsection{$C^{1,\alpha}$ regularity}

Assume that $\psi$ is convex and satisfies $\psi<\varphi$ on $\partial\Omega$. Let $u$ be the maximizer of \eqref{e204} and $\mathcal{G}_u$ be the graph of $u$ over $\Omega$. From Section 4 we know $\mathcal{G}_u$ is strictly convex. The $C^{1,\alpha}$ estimate for strictly convex solutions of Monge-Amp\`ere equations was obtained by Caffarelli \cite{Caf91}. Here we adopt a similar argument from \cite{TWma}.

For an arbitrary point on $\mathcal{G}_u$, by choosing appropriate coordinates and a rotation in $\mathbb{R}^{n+1}$, we assume it is the origin and $\mathcal{G}_u\subset\{x_3\geq0\}$, and near the origin $\mathcal{G}_u$ is the graph of a strictly convex function $u$. 

\begin{lemma}\label{l501}
There exist positive constants $\alpha,\beta$, and $C$ such that
	\begin{equation}\label{e501}
		C^{-1}|x|^{1+\beta}\leq u(x)\leq C|x|^{1+\alpha}\quad \mbox{near the origin.}
	\end{equation}
\end{lemma}

\begin{proof}
Denote $S_h^0=\{x\in\Omega\,:\,u(x)<h\}$. By the strict convexity, $S_h^0\Subset\Omega$ when $h>0$ is small. We point out that the proof of strict convexity in Section 4 implies that $u$ is $C^1$ smooth. In fact, if $u$ is not $C^1$ at some point, by a rotation of axes we assume $\mathcal{G}_u\subset\{x_3\geq a|x_1|\}$ for some constant $a>0$. Let $L$ be the intersection of $\mathcal{G}_u$ with $\{x_3=0\}$. $L$ could be a single point or a segment on $x_2$-axis. From the proof of Lemma \ref{l402}, by a contradiction argument, we can rule out the possibility of both cases, which implies that $\mathcal{G}_u$ is $C^1$ smooth. 
Hence we have 
	\begin{equation}\label{e502}
		\dist\left(S_{h/2}^0, \partial S_h^0\right)\geq C_1,
	\end{equation}
or equivalently,
	\begin{equation}\label{e503}
		u(\theta x)\geq\frac12u(x)
	\end{equation}
for any $x\in\partial S_h^0$, where $\theta=1-\frac12C_1$. As $h$ is any small constant, it follows that for any $x$ near the origin,
	\begin{equation}\label{e504}
		u(x)\geq 2^{-k}u(\theta^{-k}x)
	\end{equation}
provided $\theta^{-k}x\in\Omega$. Hence we obtain the first inequality in \eqref{e501} with $\beta$ given by $\theta^{1+\beta}=1/2$.

To prove the second inequality, we claim that there exists a constant $\sigma>0$ such that for any small $h>0$ and any $x\in\partial S_h^0$,
	\begin{equation}\label{e505}
		u(\frac12x)<\frac{1-\sigma}{2}u(x).
	\end{equation}
Define $\alpha$ by $1-\sigma=2^{-\alpha}$. Then for any $x\in\partial\Omega$ and any $t\in(\frac{1}{2^{k+1}},\frac{1}{2^k})$,
	\begin{equation}\label{e506}
	\begin{split}
		u(tx) &\leq 2^{-k}(1-\sigma)^ku(x) \\
			&= (2^{-k})^{1+\alpha}u(x) \\
			&\leq 2t^{1+\alpha}u(x).
	\end{split}
	\end{equation}
Hence $u\in C^{1,\alpha}$.

Inequality \eqref{e505} follows from \eqref{e503} as proved in \cite{TWma}. For the reader's convenience, we include it here. 
Consider the convex function $g(t)=u(tx)$, $t\in[-1,1]$. Replacing $g$ by $g/g(1)$, we may assume that $g(1)=1$. Let $\psi(t)=g(t+\frac12)-g'(\frac12)t-g(\frac12)$. Then $\psi(0)=0, \psi\geq0$. If $g(\frac12)>\frac{1-\varepsilon}{2}$, by convexity we have $1+\varepsilon\geq g'(\frac12)\geq 1-\varepsilon$ and $\psi(-\frac12)\leq\varepsilon$. Applying \eqref{e503} to $\psi$, we have $\psi(-\frac12\theta^{-1})\leq 2\psi(-\frac12)\leq 2\varepsilon$. Hence $g(-\frac12\theta^{-1}+\frac12)<0$ when $\varepsilon<\frac{1-\theta}{5}$, we reach a contradiction as $u\geq0$.
\end{proof}

We remark that the estimate \eqref{e501} was also obtained in \cite{LTW} for strictly $c$-convex solutions of general Monge-Amp\`ere equations arising in the optimal transportation by a duality argument.

\subsection{$C^{1,1}$ regularity}

Assume that $\psi$ is uniformly convex. Denote $T=\{x\in\Omega\,:\,u(x)=\psi(x)\}$ and $F=\Omega-T$. Let $\mathcal{G}_{T},\mathcal{G}_{F}$ be the graph of $u$ over $T,F$, respectively. 
For any point $p\in\partial\mathcal{G}_F$, we may choose a proper coordinate system such that $p$ is the origin; and by a rotation in $\mathbb{R}^{n+1}$, we may also assume that $\{x_{3}=0\}$ is a tangent plane of $\mathcal{G}_\psi$. Therefore, $\psi(0)=0,D\psi(0)=0$, $u\geq\psi$ and $\psi$ is uniformly convex. 

\begin{lemma}\label{l502}
Assume that $\psi$ is uniformly convex. There exist two positive constants $C_1,C_2>0$ such that
	\begin{equation}\label{e507}
		C_1|x|^2\leq u(x)\leq C_2|x|^2.
	\end{equation}
\end{lemma}

\begin{proof}
The first inequality follows from the uniform convexity of $\psi$. That is
	\[u(x)\geq\psi(x)\geq C_1|x|^2\]
as $\{x_{3}=0\}$ is the tangent plane of $\mathcal{G}_\psi$ at the origin. 

For the second inequality, suppose by contradiction that it is not true, then there is a sequence of points $x_k$ with $|x_k|\to0$ such that $u(x_k)\geq 2^k|x_k|^2$. We claim that
	\begin{equation}\label{e508}
		|N_u(E_{\varepsilon_k})|\geq C2^{k/2}\varepsilon_k^{n/2}
	\end{equation}
where $\varepsilon_k=u(x_k), E_\varepsilon=\{x\in\Omega\,:\,u(x)<\varepsilon\}$.
To prove \eqref{e508}, by a rescaling
	\[u\to\varepsilon_k^{-1}u,\quad\mbox{and }x\to\varepsilon_k^{-1/2}x,\]
we may assume $\varepsilon=1$. Let $v$ be a convex function defined on the entire $\mathbb{R}^2$ such that $v(0)=0, v=u=1$ on $\partial E_1=\partial\{u<1\}$, and $v$ is homogeneous of degree $1$. Then the graph of $v$ is a convex cone with vertex at the origin. By the convexity of $u$ we have 
	\[N_v(E_1)\subset N_u(E_1).\]
By the first inequality \eqref{e507}, we have
	\[N_v(E_1)\supset B_{C_1^{1/2}}(0),\]
the ball of radius $C_1^{1/2}$. 
By the assumption that $1=v(x_k)=u(x_k)>2^k|x_k|^2$, the slope of $v$ at $x_k$ is greater than $2^{k/2}$. Hence there exists a point $\hat p\in N_v(E_1)$ such that $|\hat p|\geq 2^{k/2}$.
Finally noting that $N_v(E_1)=N_v(\mathbb{R}^2)$ is a convex set as $v$ is a convex cone, we obtain
	\[|N_v(E_1)|\geq CC_1^{(n-1)/2}|\hat p|\geq C2^{k/2}.\]
By rescaling back, we then obtain $|N_u(E_{\varepsilon_k})|\geq C2^{k/2}\varepsilon_k^{n/2}$.
	
On the other hand, by the first inequality in \eqref{e507} we have $|E_\varepsilon|\leq C\varepsilon^{n/2}$.
Hence by the determinant estimate in \S2.5 we have
	\[|N_u(E_{\varepsilon_k})|=\int_{E_{\varepsilon_k}}\det\,D^2u\leq C\varepsilon_k^{n/2}.\]
When $k$ is sufficiently large, we reach a contradiction. 
\end{proof}

\begin{corollary}\label{c501}
There is no line segment on $\mathcal{G}_F$ with an endpoint on $\partial\mathcal{G}_F$. 
\end{corollary}

Now we prove the second part of Theorem \ref{t101}. 

\begin{theorem}\label{t501}
Suppose that $\psi$ is uniformly convex. Then $u$ is $C^{1,1}$ smooth in a neighborhood of $\partial F$. 
\end{theorem}

\begin{proof}
When $\alpha=\frac{1}{n+2}$, the $C^{1,1}$ regularity was obtained in \cite{STW} for enclosed convex hypersurfaces with maximal affine area, where the affine invariant property plays a crucial role. But for general $0\leq\alpha\leq\frac{1}{n+2}$, we need to rotate the graph $\mathcal{G}$ in $\mathbb{R}^{n+1}$ and use the a priori determinant estimates in Section 2. Note that the dimension two is needed in the proof of strict convexity, see Lemmas \ref{l402} and \ref{l404}.

Let $p=(p_1,p_2,p_{3})$ be a point on $\mathcal{G}_F$, close to $\partial\mathcal{G}_F$.
Let $\delta=\dist(p,\partial\mathcal{G}_F)$ (Euclidean distance). Choosing a proper coordinate system we suppose the origin is a point on $\partial\mathcal{G}_F$ and $|p|=\delta$.
By a rotation transform, suppose furthermore that $\mathcal{G}_\psi\subset\{x_{3}\geq 0\}$, and near the origin $u$ satisfies \eqref{e507}. 

Let $u_\delta(x)=\delta^{-2}u(\delta x)$ and let $p_\delta=\left(\frac{p_1}{\delta},\frac{p_2}{\delta},\frac{p_{3}}{\delta^2}\right)$. Then by \eqref{e507},
	\begin{equation}\label{e509}
		C_1|x|^2\leq u_\delta(x)\leq C_2|x|^2.
	\end{equation}
From Section 4, $u_\delta$ is strictly convex near $p_\delta$. 
By the a priori estimates in Section 2 and the approximation in Section 3, we then infer that there exist constants $C_1,C_2>0$ such that
	\[C_1I\leq D^2u_\delta(\bar p)\leq C_2I\]
for any $\bar p$ near $p_\delta$, where $I$ is the unit matrix. The constants $C_1$ and $C_2$ are independent of $\delta$. 
By our rescaling, $D^2u(p)=D^2u_\delta(p_\delta)$. Hence the second derivatives of $u$ are uniformly bounded near $\partial F$. This complete the proof.
\end{proof}


\end{document}